\documentclass{amsart}
\usepackage{amsfonts,amsmath,amsthm,amssymb}
\usepackage[utf8]{inputenc}
\usepackage{mathtools}
\usepackage{booktabs}
\usepackage{array}
\usepackage{paralist}
\usepackage{hyperref}

\usepackage{frenchineq}

\theoremstyle{remark}

\theoremstyle{definition}

\theoremstyle{plain}
\newtheorem{lemma}{Lemma}

\theoremstyle{plain}
\newtheorem{theorem}{Theorem}
\newtheorem{proposition}[theorem]{Proposition}

\newcommand{\Pcal}{\mathcal{P}}
\newcommand{\Ppriority}{\mathcal{P}_\textsf{priority}}
\newcommand{\Pconflict}{\mathcal{P}_\textsf{conflict}}
\newcommand{\Qcal}{\mathcal{Q}}
\newcommand{\Qsync}{\mathcal{Q}_\textsf{sync}}
\newcommand{\R}{\mathbb{R}}
\newcommand{\Rplus}{\mathbb{R}_{\geq 0}}
\newcommand{\Rsplus}{\mathbb{R}_{> 0}}
\newcommand{\N}{\mathbb{N}}

\newcommand{\G}{\mathbb{G}}
\newcommand{\out}{^{\textrm{out}}}
\newcommand{\outp}{^{\textrm{out}}_+}
\newcommand{\outm}{^{\textrm{out}}_-}
\newcommand{\inc}{^{\textrm{in}}}

\newcommand{\taucyc}{\bar \tau}

\newcommand{\ie}{i.e.}
\newcommand{\pigen}[1]{\pi_{\textrm{#1}}}
\newcommand{\taugen}[1]{\tau_{\textrm{#1}}}
\newcommand{\pie}{\pigen{ext}}
\newcommand{\piu}{\pigen{ur}}
\newcommand{\pia}{\pigen{adv}}
\newcommand{\taue}{\taugen{ext}}
\newcommand{\tauu}{\taugen{ur}}
\newcommand{\taua}{\taugen{adv}}
\newcommand{\taut}{\taugen{tr}}
\newcommand{\trans}[1][]{\stackrel{#1}{\longrightarrow}}
\newcommand{\transshort}[1][]{\rightarrow^{#1}}
\newcommand{\cadlag}{{c\`adl\`ag}}
\newcommand{\power}[2]{#2#1}
\newcommand{\eps}{\varepsilon}

\begin{document}

\title[Performance Evaluation of an Emergency Call Center]{Performance Evaluation of an Emergency Call Center:
Tropical Polynomial Systems applied to Timed Petri Nets
}
\author[X.Allamigeon]{Xavier Allamigeon} 
\author[V.B\oe{}uf]{Vianney B\oe{}uf}
\author[S.Gaubert]{St\'ephane Gaubert}
\address[X.~Allamigeon, S.~Gaubert]{INRIA and CMAP, \'Ecole polytechnique, CNRS}
\address[V.~B\oe{}uf]{\'Ecole des Ponts ParisTech, INRIA and CMAP, \'Ecole polytechnique, CNRS, Brigade de sapeurs-pompiers de Paris}
\email{FirstName.LastName@inria.fr}
\thanks{The three authors were partially supported by the programme ``Concepts, Systèmes et Outils pour la Sécurité Globale'' of the French National Agency of Research (ANR), project ``DEMOCRITE'', number ANR-13-SECU-0007-01. The first and last authors were partially supported by the programme ``Ingénierie Numérique \& Sécurité'' of ANR, project ``MALTHY'', number ANR-13-INSE-0003.}
\thanks{
A conference proceeding version of this article can be found in~\cite{ABG2015}.
}

\begin{abstract}
We analyze a timed Petri net model of an emergency call center which processes calls with different levels of priority. The counter variables of the Petri net represent the cumulated number of events as a function of time. We show that these variables are determined by a piecewise linear dynamical system. We also prove that computing the stationary regimes of the associated fluid dynamics reduces to solving a polynomial system over a tropical (min-plus) semifield of germs. This leads to explicit formul\ae{} expressing the throughput of the fluid system as a piecewise linear function of the resources, revealing the existence of different congestion phases. Numerical experiments show that the analysis of the fluid dynamics yields a good approximation of the real throughput.
\end{abstract}

\maketitle

\section{Introduction}

\subsection*{Motivations.}

Emergency call centers must handle complex and diverse help requests, 
involving different instruction procedures leading to the
engagement of emergency means. An important issue is the performance evaluation
of these centers. One needs in particular to estimate the dependence of
quantities like throughputs or waiting times
with respect to the allocation of resources, like the operators answering
calls.

The present work originates from a case study relative
to the current project led by Pr\'efecture de Police de Paris (PP),
involving the Brigade de sapeurs-pompiers de Paris (BSPP),
of a new organization to handle emergency calls to Police (number 17),
Firemen (number 18), and untyped emergency calls
(number 112), in the Paris area. 
In addition to the studies and experimentation already carried
out by PP and BSPP experts,
we aim at developing formal methods, based on mathematical models.
One would like to derive analytical formul\ae\,
or performance bounds allowing one to confirm the results
of simulation, to identify exceptional
situations not easily accessible to simulations,
and to obtain a general understanding of potential
bottlenecks. 
In such applications, complex concurrency phenomena (available operators
must share their time between different types of requests)
are arbitrated by priority rules.
The systems under study are beyond the known exactly solvable
classes of Markov models, and it is desirable to develop new analytical
results. 

\subsection*{Contributions.}

We present an algebraic approach which allows to analyze the performance of systems involving priorities and modeled by timed Petri nets. 
Our results apply to the class of Petri nets in which the places
can be partitioned in two categories:
the routing in certain places is subject to priority rules,
whereas the routing at the other places is free choice.

Counter variables determine the number of firings of the different
transitions as a function of time.
Our first result shows that, for the earliest firing rule,
the counter variables are the solutions of a
piecewise linear dynamical system (Section~\ref{sec:piecewise}). Then, we introduce a fluid
approximation in which the counter variables are real
valued, instead of integer valued.
Our main result shows that in the fluid model, the stationary regimes
are precisely the solutions of a set of lexicographic piecewise
linear equations, which constitutes a polynomial system over a tropical (min-plus) semifield of germs (Section~\ref{sec:stationary}). The latter is a modification of the ordinary tropical semifield.
In essence, our main result shows that computing stationary regimes reduces to solving tropical polynomial systems.

Solving tropical polynomial systems is one of the most basic
problems of tropical geometry. The latter provides insights on the nature of solutions, as well as algorithmic tools. In particular, the tropical
approach allows one to determine the different congestion phases of the system. 

We apply this approach to the case study of PP and BSPP.
We introduce a simplified model of emergency call center (Section~\ref{sec:model}). This allows us to concentrate on the analysis of an essential feature of the
organization: the two level emergency procedure. Operators at level $1$
initially receive the calls, qualify their urgency, handle the non urgent
ones, and transfer the urgent cases to specialized
level $2$ operators who complete the instruction. 
We solve the associated system of tropical polynomial equations
and arrive at an explicit computation of the different
congestion phases, 
depending on the ratio $N_2/N_1$
of the numbers of operators of level $2$ and $1$ (Section~\ref{sec:application}). Our analytical
results are obtained only for the approximate fluid model. However,
they are confirmed by simulations in which the original semantics of the Petri nets
(with integer firings) is respected (Section~\ref{sec:experiments}). 

\subsection*{Related work.} 
Our approach finds its origin in the maxplus modeling of timed discrete event systems, introduced by Cohen, Quadrat and Viot and further developed by Baccelli and Olsder, see~\cite{bcoq,how06} for background.
The idea of using counter variables already appeared in their work.
However, the classical results only apply to restricted classes
of Petri nets, like event graphs, or event graphs with weights
as, for instance, in recent 
work by Cottenceau, Hardouin and Boimond~\cite{cottenceau}. 
The modeling of more general Petri nets by a combination of min-plus
linear constraints and classical linear constraints was
proposed by Cohen, Gaubert and Quadrat~\cite{CGQ95b,CGQ95a}
and Libeaut and Loiseau (see~\cite{Libeaut}).
The question of analyzing the behavior of the dynamical systems
arising in this way was stated in a compendium of open problems in control theory~\cite{maxplusblondel}. 
A key discrepancy with the previously developed min-plus algebraic models
lies in the {\em semantics} of the Petri nets. The model
of~\cite{CGQ95b,CGQ95a} requires the routing to be based
on open loop {\em preselection} policies of tokens at places, 
and it does not allow for priority rules. 
This is remedied in the present
work: we show that priority rules can be written in 
a piecewise linear way, leading to a rational tropical dynamics.

Our approach is inspired by a work of Farhi, Goursat and Quadrat~\cite{farhi}, who developed a min-plus model for a road traffic network. The idea
of modeling priorities by rational min-plus dynamics 
first appeared there. By comparison, one aspect of novelty of the 
present approach consists in showing
that this idea applies to a large class of Petri nets, mixing
free choice and priority routing, so that its scope
is not limited to a special class of road traffic models. Moreover, we provide a complete proof that these Petri nets follow the rational tropical dynamics, based on a precise analysis of the counter variables along an execution trace. 
Finally,
the approach of~\cite{farhi} was developed in the discrete time
case. A novelty of the present work consists in the treatment of
the continuous time. This
requires the introduction of a symbolic perturbation
technique, working with semifield of germs. This technique
was used in~\cite{gg0} for algorithmic purposes.
It has been recently applied by Allamigeon,
Fahrenberg, Gaubert, Katz, and Legay to the analysis of timed systems~\cite{uli2013}. 

The present piecewise affine dynamical systems bear some general resemblance
with min-plus models of cellular automata, and in particular with the 
ultradiscrete Toda equation studied by Inoue and Iwao~\cite{inoue}. We believe these aspects are worth being further studied.

The analysis of timed Petri nets is a major question, which has been extensively studied. We refer to~\cite{BerthomieuDiaz91,AbdullaNylen01,Gardey04,Jacobsen11} for a non-exhaustive account on the topic, and to~\cite{tina,romeo,tapaal} for examples of tools implementing these techniques. An important effort has been devoted to the comparison of timed Petri nets with timed automata in terms of expressivity, see for instance~\cite{Berard05,Srba08}. The approaches developed in the aforementioned works aim at checking whether a given specification is satisfied (for instance, reachability, or more generally, a property expressed in a certain temporal logic), or at determining whether two Petri nets are equivalent in the sense of bisimulation. Hence, the emphasis is on issues different from the present ones: we focus on the performance analysis of timed Petri nets, by determining the asymptotic throughputs of transitions.

\subsubsection*{Acknowledgments.} We thank R\'egis Reboul, from PP, in charge of the emergency call centers project, and Commandant St\'ephane Raclot, from BSPP, for the information and insights they provided throughout the present work. We are grateful to the anonymous reviewers for their comments which helped to improve the presentation of this paper.  

\section{A simplified Petri net model of an emergency call center}\label{sec:model}

\begin{figure}[t] 
\begin{center}
\includegraphics{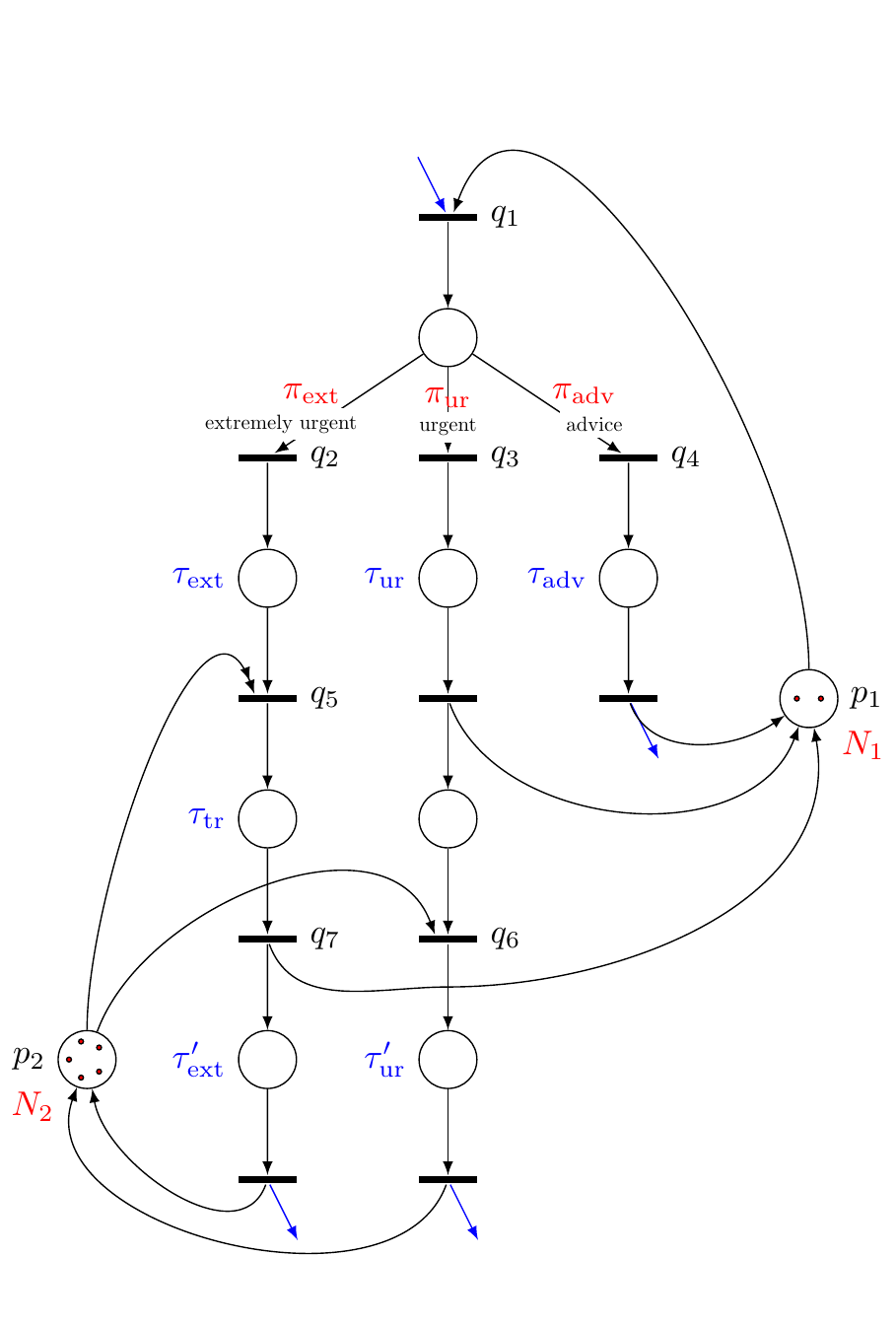}
\end{center}
\caption{Simplified Petri net model of the Parisian 17-18-112 emergency call center (organization in project). Blue arrows do not belong to the Petri net and symbolize the entrance and exit of calls in the system.}\label{fig:Petri_Net}
\end{figure}

In this section, we describe a call center
answering to emergency calls according to a two level
instruction procedure.
In the new organization planned by PP together with BSPP~\cite{reboulraclot}, 
the emergency calls to the
police (number 17), to the firemen (18), and untyped emergency calls
(European number 112)  will be dealt with according to a unified
procedure, allowing a strong coordination. Another important
feature of this organization is that it involves a two level treatment.
In the present paper, we limit our attention to the analysis of the two level procedure.
We defer to a further work the analysis of the unification of the treatment of calls with heterogeneous characteristics.
Hence, we discuss a simplified model, for academic purposes.

The first level operators filter the calls and assign them to
three categories: extremely urgent (potentially life threatening situation),
urgent (needing further instruction), and non urgent (e.g., 
call for advice). Non-urgent calls are dealt with entirely by level~1 operators.
Extremely urgent and urgent calls are passed to level~2 operators.
An advantage of this procedure lies in robustness considerations. 
In case of events generating bulk calls, the access to level~2 
experts is protected by the filtering of level~1. This allows
for better guarantees of service for the extremely urgent calls.
Every call qualified as extremely urgent
generates a 3-way conversation: the level~1 operator
stays in line with the calling person when the call is passed to the
level~2 operator.
Such 3-way conversations allow one to avoid any loss of information,
and were shown to contribute
to the quality of the procedure~\cite{reboulraclot}.
Proper dimensioning of resources is needed to make sure that the
synchronizations between level~1 and level~2 operators created
by these 3-way conversations do not create bottlenecks.
We focus on the case where the system is saturated, that is, there is an infinite queue of calls that have to be handled. We want to evaluate the performance of the system, \ie~the throughput of treatment of calls by the operators.

The call center is modeled by the timed Petri net
of Figure~\ref{fig:Petri_Net}. We describe here the net
in informal terms, referring the reader to Section~\ref{sec:piecewise}
for more information on Petri nets and the semantics that we adopt.
We use the convention that all transitions can be fired instantaneously. Holding times are attached to places. 

Let us give the interpretation in terms of places and transitions. 
The number of operators of level $1$ and $2$ is 
equal to $N_1$ and $N_2$, respectively. The marking in
places $p_1$ and $p_2$, respectively, represents the number
of idle operators of level $1$ or $2$ at a given time.
In particular, the number of tokens initially available
in places $p_1$ and $p_2$ is $N_1$ and $N_2$.
The initial marking of other places is zero. 
A firing of transition $q_1$ represents the beginning
of a treatment of an incoming emergency call by a level~1 operator.
The arc from place $p_1$ to transition $q_1$ indicates that every call requires
one level $1$ operator. The routing from transition $q_1$ to transitions
$q_2,q_3,q_4$  represents the qualification of a call as extremely urgent, urgent, or non urgent (advice). The proportions of these calls are denoted by $\pie$, $\piu$, and $\pia$, respectively, so that $\pie + \piu + \pia = 1$. The proportions
are known from historical data. The instruction of the call
at level $1$ is assumed to take a deterministic time $\taue$, $\tauu$, or $\taua$, respectively, depending on the type of call. 

After the treatment of a non urgent or urgent call at level~1, the level~1 operator is made immediately available to handle a new call. This is represented by the arcs leading
to place $p_1$ from the transitions located below the places with
holding times $\tauu$ and $\taua$. Before an idle operator of level~2 is assigned to the treatment of an urgent call, which is represented by the firing of transition $q_6$, the call is stocked in the place located above $q_6$. In contrast, the sequel of the processing
of an extremely urgent call (transition $q_5$) requires the availability of a level $2$ operator (incoming arc $p_2 \to q_5$) in order to initiate a $3$-way conversation. The level $1$ operator
is released only after a time $\taut$ corresponding to the duration of this conversation. This is represented by the arc $q_7\to p_1$. The double arrow depicted on the arc $p_2 \to q_5$ means that level~2 operators are assigned to the treatment of extremely urgent calls (if any) in priority. 
The holding times $\taue'$ and $\tauu'$ represent the time needed by a level $2$ operator to complete the instruction of extremely urgent and urgent calls respectively. 

\section{Piecewise linear dynamics of timed Petri nets with free choice and priority routing}\label{sec:piecewise}

\subsection{Timed Petri nets: notation and semantics}

A \emph{timed Petri net} consists of a set $\Pcal$ of places and a set $\Qcal$ of transitions, in which each place $p \in \Pcal$ is equipped with a holding time $\tau_p \in \Rsplus$ as well as an initial marking $M_p \in \N$. Given a place $p \in \Pcal$, we respectively denote by $p\inc$ and $p\out$ the sets of input and output transitions. Similarly, for all $q \in \Qcal$, the sets of upstream and downstream places are denoted by $q\inc$ and $q\out$ respectively.

The semantics of the timed Petri net which we use in this paper is based on the fact that every token entering a place $p \in \Pcal$  must stay at least $\tau_p$ time units in place $p$ before becoming available for a firing of a downstream transition. 
More formally, a state of the semantics of the Petri net specifies, for each place $p \in \Pcal$, the set of tokens located at place $p$, together with the age of these tokens since they have entered place $p$. 
In a given state $\sigma$, the Petri net can evolve into a new state $\sigma'$ in two different ways: 
\begin{asparaenum}[(i)]
\item either a transition $q \in \Qcal$ is fired, which we denote $\sigma \trans[q] \sigma'$. This occurs when every upstream place $p$ contains a token whose age is greater than or equal to $\tau_p$. The transition is supposed to be instantaneous. A token enters in each downstream place, and its age is set to $0$;
\item or all the tokens remain at their original places, and their ages are incremented by the same amount of time $d \in \Rplus$. This is denoted $\sigma \trans[d] \sigma'$. 
\end{asparaenum}

In the initial state $\sigma^0$, all the tokens of the initial marking are supposed to have an ``infinite'' age, so that they are available for firings of downstream transitions from the beginning of the execution of the Petri net. 
The set of relations of the form $\trans[q]$ and $\trans[d]$ constitutes a timed transition system which, together with the initial state $\sigma^0$, fully describe the semantics of the Petri net. 
Note that in this semantics, transitions can be fired simultaneously. In particular, a given transition can be fired several times at the same moment. Recall that every holding time $\tau_p$ is positive, so that we cannot have any Zeno behavior.

In this setting, we can write any execution trace of the Petri net as a sequence of transitions of the form:
\begin{equation}
\sigma^0 \trans[d^0] \; \trans[q^0_1] \; \trans[q^0_2] \dots  \trans[q^0_{n^0}] \sigma^1 \trans[d^1] \; \trans[q^1_1] \; \trans[q^1_2] \dots\trans[q^1_{n^1}] \sigma^2 \trans[d^2] \dots \label{eq:trace}  
\end{equation}
where $d^0 \geq 0$ and $d^1, d^2, \dots > 0$. In other words, we consider traces in which we remove all the time-elapsing transitions of duration $0$, except the first one, and in which time-elapsing transitions are separated by groups of firing transitions occurring simultaneously. We say that a transition $q$ is \emph{fired at the instant $t$} if there is a transition $\trans[q]$ in the trace such that the sum of the durations of the transitions of the form $\trans[d]$ which occur before in the trace is equal to $t$. The \emph{state of the Petri net at the instant $t$} refers to the state of the Petri net appearing in the trace~\eqref{eq:trace} after all transitions have been fired at the instant $t$.

In the rest of the paper, we stick to a stronger variant of the semantics, referred to as \emph{earliest behavior} semantics, in which every transition $q$ is fired at the earliest moment possible. More formally, this means that in any state $\sigma$ arising during the execution, a place $p$ is allowed to contain a token of age (strictly) greater than $\tau_p$ only if no downstream transition can be fired (\ie~no transition $\trans[q]$ with $q \in p\out$ can be applied to $\sigma$). The motivation to study the earliest behavior semantics originates from our interest for emergency call centers, in which all calls are supposed to be handled as soon as possible.

\subsection{Timed Petri nets with free choice and priority routing}

In this paper, we consider timed Petri nets in which places are free choice, or subject to priorities. 
This class of nets includes our model of emergency call center. Recall that a place $p \in \Pcal$ is said to be \emph{free choice} if either $|p\out| = 1$, or all the downstream transitions $q \in p\out$ satisfy $q\inc = \{p\}$. 
The main property of such a place is the following: if one of the downstream transitions is activated (\ie~it can be potentially fired), then the other downstream transitions are also activated. 
A place is \emph{subject to priority} if the available tokens in this place are routed to downstream transitions according to a certain priority rule. We denote by $\Ppriority$ the set of such places.
We assume that no transition has more than one upstream place subject to priority, that is, for any transition $q$, the set $q\inc \cap \Ppriority$ has at most one element.
This allows to avoid inconsistency between priority rules (e.g.~two priority places acting on the same transitions in a contradictory way).
For the sake of simplicity, we also assume in the following that every $p \in \Ppriority$ has precisely two downstream transitions, which we respectively denote by $p\outp$ and $p\outm$. Then, if both transitions are activated, the tokens available in place $p$ are assigned to $p\outp$ as a priority. Equivalently, in the execution trace of the Petri net, we have $\sigma \transshort[p\outm] \sigma'$ only if the transition $\transshort[p\outp]$ cannot be applied to the state $\sigma$. We remark that it is possible to handle multiple priority levels, up to making the presentation of the subsequent results more complicated.

\begin{figure}[t]
\begin{center}
\includegraphics{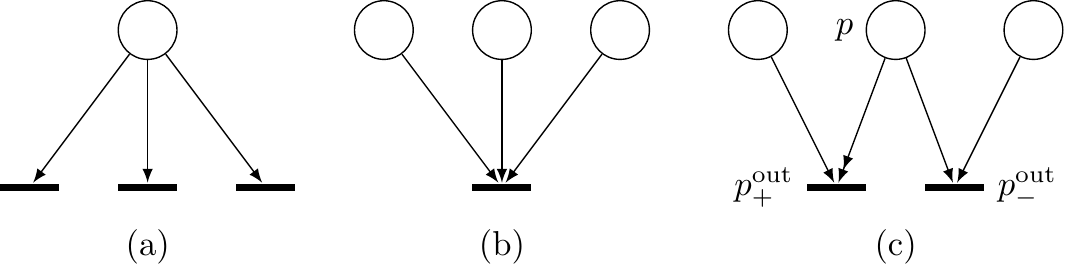}
\end{center}
\caption{Conflict, synchronization and priority configurations.}\label{fig:configuration}	
\end{figure}

To summarize, there are three possible place/transition patterns which can occur in the timed Petri nets that we consider, see Figure~\ref{fig:configuration}. 
The first two ones involve only free choice places, and are referred to as \emph{conflict} and \emph{synchronization} patterns respectively. 
We denote by $\Pconflict$ the set of free choice places that have at least two output transitions, and by $\Qsync$ the set of transitions such that every upstream place $p$ satisfies $|p\out| = 1$. By definition, we have $\Pconflict \cap (\Qsync)\inc = \emptyset$. 
The third configuration in Figure~\ref{fig:configuration} depicts a place $p$ subject to priority. 
In order to distinguish $p\outp$ and $p\outm$, we depict the arc leading to the transition $p\outp$ by a double arrow.
By assumption, the places $r \neq p$ located upstream $p\outp$ and $p\outm$ are non-priority, so that they are free-choice and have only one output transition, as depicted in Figure~\ref{fig:configuration}(c). 

\subsection{Piecewise linear representation by counter variables}\label{sub:counterDynamics}

Since we are interested in estimating the throughput of transitions in a Petri net, we associate with any transition $q \in \Qcal$ a counter variable $z_q$ from $\R$ to $\N$ such that $z_q(t)$ represents the number of firings of transition $q$ that occurred up to time $t$ included. 
Similarly, given a place $p \in \Pcal$, we denote $x_p(t)$ the number of tokens that have entered  place $p$ up to time $t$ included. Note that the tokens initially present in place $p$ are counted. More formally, $x_p(t)$ is given by the sum of the initial marking $M_p$ and of the numbers of firings of transitions $q \in p\inc$ which occurred before the instant $t$ (included). 
We extend the counter variables $x_p$ and $z_q$ to $\R_{< 0}$ by setting:
\begin{equation}
x_p(t) = M_p \, , \quad 
z_q(t) = 0 \, , \quad
\text{for all} \; t < 0
\enspace .
\label{eq:init}
\end{equation}

By construction, the functions $x_p$ and $z_q$ are non-decreasing. Besides, since they count tokens up to time $t$ \emph{included}, they are \emph{\cadlag} functions, which means that they are right continuous and have left limits at any point. 
Given a {\cadlag} function $f$, we denote by $f(t^-)$ the left limit at the point $t$.

The goal of this section is to describe the dynamics of timed Petri nets with free choice and priority routing by means of a set of piecewise linear equality constraints over the counter variables. We provide an informal presentation of these constraints. First observe that we necessarily have: 
\begin{equation}
\forall p \in \Pcal \,, \quad x_p(t) =  M_p  + \sum_{q\in p\inc} z_q(t) \, , \label{eq:pnpriority1}
\end{equation}
as the initial marking $M_p$ is counted in $x_p(t)$, and any token entering place $p$ 
before the instant $t$ must have been fired from an upstream transition $q \in p\inc$ before. 
In a similar way, if $p \in \Pconflict$, the total number of times the downstream transitions have been fired before the instant $t$ is necessarily equal to the number of tokens which entered place $p$ before time $t - \tau_p$ (included). This is due to the fact that if a token enters $p$ at the instant $s$, then it is consumed \emph{exactly} at the instant $s + \tau_p$ (by definition of the earliest behavior semantics). This yields the identity:
\begin{equation}
\forall p \in \Pconflict \,, \quad \sum_{q \in p\out} z_q(t) = x_p(t - \tau_p) \, .\label{eq:pnpriority2}
\end{equation}
Now consider a transition $q \in \Qsync$. The number of times this transition is fired at the instant $t$ is given by $z_q(t) - z_q(t^-)$. In each upstream place $p \in q\inc$, the number of tokens which are available for firing $q$ is equal to $x_p(t - \tau_p) - z_q(t^-)$. Indeed, since  place $p$ does not
have any other output transition, the total number of tokens which have left place $q$ until the instant $t$ equals $z_q(t^-)$. 
By definition of the earliest behavior semantics, the number of firings of $q$ at the instant $t$ must be exactly equal to the minimum number of tokens available in places $p \in q\inc$. If we denote $\min(x,y)$ by $x \wedge y$, we consequently get:
\begin{equation}
\forall q \in \Qsync \,, \quad z_q(t) = \bigwedge_{p \in q\inc} x_p(t - \tau_p) \, .\label{eq:pnpriority3}
\end{equation}
Finally, let us take a place $p \in \Ppriority$. Since the transition $p\outp$ has priority over $p\outm$, the quantity $z_{p\outp}(t) - z_{p\outp}(t^-)$ must be equal to the minimal number of tokens available in the upstream places, including $p$. For every place $r \in (p\outp)\inc$ distinct from $p$, the number of available tokens is given by $x_r(t - \tau_r) - z_{p\outp}(t^-)$ (recall that $p\outp$ is the only downstream transition of $r$). In contrast, the number of tokens available for firing in place $p$ is equal to $x_p(t - \tau_p) - (z_{p\outp}(t^-) + z_{p\outm}(t^-))$. We deduce that we have:
\begin{equation}
\forall p \in \Ppriority \,, \quad z_{p\outp}(t) = \bigl(x_p(t-\tau_p) - z_{p\outm}(t^-) \bigr) \wedge \bigwedge_{\substack{r \in (p\outp)\inc\\ r \neq p}} x_r(t - \tau_r) \, .\label{eq:pnpriority4}
\end{equation}
The number of tokens from place $p$ which are available for the transition $p\outm$ after the firings of $p\outp$ is given by $x_p(t - \tau_p) - (z_{p\outp}(t^-) + z_{p\outm}(t^-)) - (z_{p\outp}(t) - z_{p\outp}(t^-))$. Hence, we obtain:
\begin{equation}
\forall p \in \Ppriority \,, \quad z_{p\outm}(t) = \bigl(x_p(t-\tau_p) - z_{p\outp}(t) \bigr) \wedge \bigwedge_{\substack{r \in (p\outm)\inc\\ r \neq p}} x_r(t - \tau_r) \, .\label{eq:pnpriority5}
\end{equation}
We summarize the previous discussion by the following result:
\begin{theorem}\label{th:semantics}
Given any execution trace of a timed Petri net with free choice and priority routing, the counter variables $x_p$ ($p \in \Pcal$) and $z_q$ ($q \in \Qcal$) satisfy the constraints~\eqref{eq:pnpriority1}--\eqref{eq:pnpriority5} for all $t \geq 0$, together with the initial conditions~\eqref{eq:init}.
\end{theorem}
We refer to Appendix~\ref{app:semantics} for a detailed proof of this statement. 
Notice that, if we do not restrict to the earliest behavior semantics, the constraints~\eqref{eq:pnpriority2}--\eqref{eq:pnpriority5} are relaxed to inequalities.

So far, we have described the dynamics of timed Petri nets in the continuous time setting. However, since the Petri net of our case study is a model of a real system which is implemented in silico, we need to investigate the dynamics in discrete time as well. 
In more details, assuming that all the quantities $\tau_p$ are multiple of an elementary time step $\delta > 0$, the discrete-time version of the semantics of the Petri net restricts the transitions $\trans[d]$ to the case where $d$ is a multiple of $\delta$. 
In this case, on top of being {\cadlag}, the functions $x_p$ and $z_q$ are constant on any interval of the form $[k \delta, (k+1) \delta)$ for all $k \in \N$. 
Then, we can verify that the following result holds:
\begin{proposition}
In the discrete time semantics, the counter variables $x_p$ and $z_q$ satisfy the constraints~\eqref{eq:pnpriority1}--\eqref{eq:pnpriority5} for all $t \geq 0$, independently of the choice of the elementary time step $\delta$.
\end{proposition}
In other words, the dynamics in continuous-time is a valid representation of the dynamics in discrete time 
which allows to abstract from the discretization time step. We also note that we can refine the constraint given in~\eqref{eq:pnpriority4} by replacing the left limit $z_{p\outm}(t^-)$ by an explicit value:
\begin{equation}
\forall p \in \Ppriority \,, \quad z_{p\outp}(t) = 
\begin{dcases}
\begin{multlined}
\bigl(x_p(t-\tau_p) - z_{p\outm}(t - \delta) \bigr) \\
\quad \wedge \bigwedge_{r \in (p\outp)\inc \, , \,  r \neq p} x_r(t - \tau_r) 
\end{multlined} & \text{if} \; t \in \delta\N \, , \\
\begin{multlined}
\bigl(x_p(t-\tau_p) - z_{p\outm}(t) \bigr) \\
\quad \wedge \bigwedge_{r \in (p\outp)\inc \, , \, r \neq p} x_r(t - \tau_r) 
\end{multlined} & \text{otherwise.}
\end{dcases}\label{eq:pnpriority4discrete}
\end{equation}
(Here and below, we denote by $\delta \N$ the set $\{0, \delta, 2\delta, \dots \}$.) The system formed by the constraints~\eqref{eq:pnpriority1}--\eqref{eq:pnpriority3}, \eqref{eq:pnpriority5}, \eqref{eq:pnpriority4discrete}  
is referred to as the \emph{$\delta$-discretization of the Petri net dynamics}.

The only source of non-determinism in the model that we consider is the routing policy in the conflict pattern (Figure~\ref{fig:configuration}(a)). 
In the sequel, we assume that the tokens are assigned according to a 
stationary probability distribution.
Given a free choice place $p \in \Pconflict$, we denote by $\pi_{qp}$ the probability that an available token is assigned to the transition $q \in p\out$. In the following, we consider a \emph{fluid approximation of the dynamics} of the system, in which the $x_p$ and $z_q$ are non-decreasing {\cadlag} functions from $\R$ to itself, and the routing policy degenerates in sharing the tokens in fractions $\pi_{qp}$. Equivalently, the fluid dynamics is defined by the constraints~\eqref{eq:pnpriority1}--\eqref{eq:pnpriority5} and the following additional constraints:
\begin{equation}
\forall p \in \Pconflict \, , \, \forall q \in p\out \,, \quad z_q(t) = \pi_{qp} x_p(t - \tau_p) \ . 
\label{eq:pnpriority6fluid}
\end{equation}
Note that the latter equation is still valid in the context of discrete time. By extension, the system formed by the constraints \eqref{eq:pnpriority1}--\eqref{eq:pnpriority3}, \eqref{eq:pnpriority5}--\eqref{eq:pnpriority6fluid} is referred to as the \emph{$\delta$-discretization of the fluid dynamics}.

\subsection{Application to our Petri net model of emergency call center}\label{ex:fluid_dynamics}
	
We illustrate Theorem~\ref{th:semantics} on the Petri net of Figure~\ref{fig:Petri_Net}.
We point out that in Figure~\ref{fig:Petri_Net}, we have omitted to specify the holding time of some places. By default, this holding time is set to a certain $\tau_\eps > 0$, and is meant to be negligible w.r.t.~the other holding times. 

For simplicity, we omit the counter variables of the places distinct from $p_1$ and $p_2$. Indeed, each of theses places $p$ has a unique input transition $q$, and its initial marking is $0$. Therefore, by definition, we have $x_p(t) = z_q(t)$ for all $t$, which means that $x_p$ can be trivially substituted in the constraints. Similarly, we omit the transitions which lead to places $p_1$ and $p_2$, as their counter variables correspond the counter variables of some transitions located upstream and shifted by the holding time of the place in between. Finally, we denote by $z_i$ the counter variables of transitions $q_i$, and by $x_i$ the counter variables of places $p_i$. We can verify that the fluid dynamics is then given by the following constraints:
\[
\allowdisplaybreaks
\begin{aligned}
z_1 (t) &= x_1(t - \tau_\eps) \\
z_2 (t) &= \power{z_1(t - \tau_\eps)}{\pie} \\
z_3 (t) &= \power{z_1(t - \tau_\eps)}{\piu} \\
z_4 (t) &= \power{z_1(t - \tau_\eps)}{\pia} \\
z_5 (t) &= (x_2(t- \tau_\eps) - z_6 (t^-)) \wedge z_2(t - \taue) \\
z_6 (t) &= (x_2(t- \tau_\eps) - z_5 (t)) \wedge z_3(t - \tauu - \tau_\eps) \\
z_7 (t) &= z_5(t - \taut)  \\
x_1 (t) &= N_1 + z_7(t) + z_3(t - \tauu) + z_4(t - \taua)\\
x_2 (t) &= N_2 + z_7(t - \taue') +z_6 (t - \tauu')
\end{aligned}
\]
They can be simplified into the following system:
\begin{equation}
\label{eqn:shortModel2}
\begin{aligned}
z_1(t) & = N_1 + z_5(t - \taut) + \piu z_1(t - \tauu - 2\tau_\eps) + \pia z_1(t - \taua - 2\tau_\eps) \\
z_5(t) & = \bigl(N_2 + z_5(t - \taut - \taue' - \tau_\eps) + z_6 (t - \tauu' - \tau_\eps) - z_6 (t^-) \bigr) \\
& \qquad \wedge \pie z_1(t - \taue -\tau_\eps)  \\
z_6(t) & = \bigl(N_2 + z_5(t - \taut - \taue' -\tau_\eps) + z_6 (t - \tauu' -\tau_\eps) - z_5 (t) \bigr) \\
& \qquad \wedge \piu z_1(t - \tauu -\tau_\eps) 
\end{aligned}
\end{equation}
which involve the counter variables $z_1$, $z_5$ and $z_6$ only. These variables correspond to the key characteristics of the system. They respectively represent the number of calls handled at level~1, and the number of extremely urgent and urgent calls handled at level~2, up to time $t$. All the other counter variables can be straightforwardly obtained from $z_1$, $z_5$ and $z_6$.

For the sake of readability, we slightly modify the original holding times $\taue$, $\tauu$, $\dots$ to incorporate the effect of $\tau_\eps$. In more details, we substitute $\taue$, $\tauu$, $\taua$, $\taue'$ and $\tauu'$ by $\taue -\tau_\eps$, $\tauu - 2\tau_\eps$, $\taua - 2\tau_\eps$, $\taue' - \tau_\eps$ and $\tauu' - \tau_\eps$ respectively. Then, System~\eqref{eqn:shortModel2} simply reads as:
\begin{equation}
\label{eqn:shortModel}
\begin{aligned}
z_1(t) & = N_1 + z_5(t - \taut) + \piu z_1(t - \tauu) + \pia z_1(t - \taua) \\
z_5(t) & = \bigl(N_2 + z_5(t - \taut - \taue') + z_6 (t - \tauu') - z_6 (t^-) \bigr) \wedge \pie z_1(t - \taue)  \\
z_6(t) & = \bigl(N_2 + z_5(t - \taut - \taue') + z_6 (t - \tauu') - z_5 (t) \bigr) \wedge \piu z_1(t - \tauu) 
\end{aligned}
\end{equation}
This is the system which we consider in the rest of the paper.

\section{Computing stationary regimes}\label{sec:stationary}

We investigate the stationary regimes of the fluid dynamics
associated with Petri nets with free choice and priority routing. More specifically, our goal is to characterize the non-decreasing {\cadlag} solutions $x_p$ and $z_q$ of the dynamics which behave ultimately as affine functions $t \mapsto u + \rho t$ ($u \in \R$ and $\rho \in \Rplus$). By {\em ultimately}, we mean that the property holds
for $t$ large enough.
In this case, the scalar $\rho$ corresponds to the asymptotic throughput of the associated place or transition. 
However, if the functions $x_p$ and $z_q$ are continuous, and a fortiori if they are affine, their values at points $t$ and $t^-$ coincide, and then, the effect of the 
priority rule on the dynamics vanishes (see Equation~\eqref{eq:pnpriority4}).
Hence, looking for ultimately affine solutions of the continuous
time equations might look as an ill-posed problem, if one
interprets it in a naive way.
In contrast, looking for the ultimately affine solutions of the 
$\delta$-discretization of the fluid dynamics is a perfectly
well-posed problem.
In other words, we aim at determining the solutions $x_p$ and $z_q$ of the discrete dynamics which coincide with affine functions at points $k \delta$ for all sufficiently large $k \in \N$. 
These solutions are referred to as the \emph{stationary solutions} of the dynamics. As we shall prove in Theorem~\ref{th:stationary}, the characterization of these solutions does not depend on the value of $\delta$, leading to a proper
definition of ultimately affine solutions of the continuous time dynamics.

In order to determine the stationary regimes, we use the notion of germs of affine functions. We introduce an equivalence relation $\sim$ over functions from $\R$ to itself, defined by $f \sim g$ if $f(t)$ and $g(t)$ are equal for all $t \in \delta \N$ sufficiently large. A {\em germ of function} (at point infinity) is an equivalence class of functions with respect to the relation $\sim$. 
For brevity, we refer to the germs of affine functions as {\em affine germs}, and we denote by $(\rho, u)$ the germ of the function $t \mapsto u + \rho t$. In this setting, our goal is to determine the affine germs of the counter variables of the Petri net in the stationary regimes. 

Given two functions $f$ and $g$ of affine germs $(\rho, u)$ and $(\rho', u')$ respectively, it is easy to show that $f(t) \leq g(t)$ for all sufficiently large $t \in \delta \N$ if, and only if, the couple $(\rho, u)$ is smaller than or equal to $(\rho', u')$ in the lexicographic order.
Moreover, the affine germ of the function $f + g$ is simply given by the germ $(\rho + \rho', u + u')$, which we denote by $(\rho, u) + (\rho', u')$ by abuse of notation. 
As a consequence, affine germs provide an ordered group. 
Let us add to this group a greatest element $\top$, with the convention that $\top +  (\rho, u) = (\rho, u) + \top = \top$. 
Then, we obtain the \emph{tropical (min-plus) semiring of affine germs} $(\G, \wedge, +)$, where $\G$ is defined as $\{\top \} \cup \R^2$, and for all $x, y \in \G$, $x \wedge y$ stands for the minimum of $x$ and $y$ in lexicographic order (extended to $\top$). 
Since in $\G$, the addition plays the role of the multiplicative law, the additive inversion defined by $-(\rho, u) := (-\rho, -u)$ corresponds to a division over $\G$. This makes $\G$ a semifield, \ie, in loose terms, a structure similar to a field, except that the additive law has no inverse. Finally, we can define the multiplication by a scalar $\lambda \in \R$ by $\lambda (\rho, u) := (\lambda \rho, \lambda u)$. 
When $\lambda \in \N$, this can be understood as an exponentiation operation in $\G$. 

Instantiating the functions $x_p$ and $z_q$ by affine asymptotics $t \mapsto u_p + t \rho_p$ and $t \mapsto u_q + t \rho_q$ in the $\delta$-discretization of the fluid dynamics leads to the following counterparts of the constraints~\eqref{eq:pnpriority1}, \eqref{eq:pnpriority3}, \eqref{eq:pnpriority5} and~\eqref{eq:pnpriority6fluid}, the variables being now elements of the semifield $\G$ of germs:
\begin{subequations}\label{eq:fixpoint_germ}
\allowdisplaybreaks
\begin{align}
\forall p \in \Pcal \, , \qquad (\rho_p, u_p) & = (0, M_p) + \sum_{q \in p\inc} (\rho_q, u_q) \label{eq:fixpoint_germ1} \\
\forall p \in \Pconflict \, , \forall q \in p\out \, , \qquad (\rho_q, u_q) & = \pi_{qp} (\rho_p, u_p - \rho_p \tau_p) 
\label{eq:fixpoint_germ2} \\
\forall q \in \Qsync \, , \qquad (\rho_q, u_q) & = \bigwedge_{p \in q\inc} (\rho_p, u_p - \rho_p \tau_p) \label{eq:fixpoint_germ3}\\
\forall p \in \Ppriority \, , \qquad 
(\rho_{p\outm}, u_{p\outm}) & = 
\begin{multlined}[t]
(\rho_p - \rho_{p\outp}, u_p - \rho_p \tau_p - u_{p\outp}) \\
\wedge 
\bigwedge_{r \in (p\outm)\inc \, , \, r \neq p} (\rho_r, u_r - \rho_r \tau_r)
\end{multlined} \label{eq:fixpoint_germ4}
\end{align}
\end{subequations}

Given $p \in \Ppriority$, the transposition of~\eqref{eq:pnpriority4} (or equivalently~\eqref{eq:pnpriority4discrete}) to germs is more elaborate due to the occurrence of the left limit $x_{p\outm}(t^-)$. We obtain:
\begin{equation} 
(\rho_{p\outp}, u_{p\outp}) = 
\begin{dcases}	
\begin{multlined}
(\rho_p - \rho_{p\outm}, u_p - \rho_p \tau_p - u_{p\outm}) \\
\wedge \bigwedge_{r \in (p\outp)\inc \, , \, r \neq p} (\rho_r, u_r - \rho_r \tau_r)
\end{multlined}
& \text{if}\; \rho_{p\outm} = 0 \, , \\
\bigwedge_{r \in (p\outp)\inc \, , \, r \neq p} (\rho_r, u_r - \rho_r \tau_r)
& \text{otherwise.} 
\end{dcases}
\tag{\ref{eq:fixpoint_germ}e}\label{eq:fixpoint_germ5}
\end{equation}

The correctness of these constraints is stated in the following result (see Appendix~\ref{app:stationary} for a detailed proof):
\begin{theorem}\label{th:stationary}
The affine germs of the stationary solutions of the $\delta$-discretization of the fluid dynamics are precisely the solutions of System~\eqref{eq:fixpoint_germ} such that $\rho_p, \rho_q \geq 0$ ($p \in \Pcal$, $q \in \Qcal$).
\end{theorem}

Since the expressions at the right hand side of the constraints of System~\eqref{eq:fixpoint_germ} involve minima of linear terms, these expressions can be interpreted as fractional functions over the tropical semifield $\G$.
In this way, System~\eqref{eq:fixpoint_germ} can be thought of as a set of tropical polynomial constraints (or more precisely, rational constraints).
 
The solutions of tropical polynomial systems is a topic of current interest, 
owing to its relations with fundamental algorithmic issues concerning
classical polynomial system solving over the reals. Here, we describe a simple method to solve  System~\eqref{eq:fixpoint_germ}, which is akin to \emph{policy search} in stochastic control. Observe that System~\eqref{eq:fixpoint_germ} corresponds to a fixpoint equation $(\rho, u) = f(\rho, u)$, where the function $f$ can be expressed as the infimum $\bigwedge_{\pi} f^{\pi}$ of finitely many linear (affine) maps $f^\pi$. In more details, every function $f^\pi$ is obtained by selecting one term for each minimum operation $\bigwedge$ occurring in the constraints (for instance, in~\eqref{eq:fixpoint_germ3}, we select one term $(\rho_p, u_p - \rho_p \tau_p)$ with $p \in q\inc$). For every selection~$\pi$, we can solve the associated linear system $(\rho, u) = f^\pi(\rho, u)$, and under some structural assumptions on the Petri net, the solution $(\rho^\pi, u^\pi)$ is unique. If $f^\pi (\rho^\pi, u^\pi) = f(\rho^\pi, u^\pi)$, \ie~in every constraint, the term we selected is smaller than or equal to the other terms appearing in the minimum, then $(\rho^\pi, u^\pi)$ forms a solution of System~\eqref{eq:fixpoint_germ} associated with the selection $\pi$. Otherwise, the selection $\pi$ does not lead to any solution. Iterating this technique over the set of selections provides all the solutions of System~\eqref{eq:fixpoint_germ}. Every iteration can be done in polynomial time. However, since there is an exponential number of possible selections, the overall time complexity of the method is exponential in the size of the Petri net.

\section{Application to the emergency call center}\label{sec:application}

We now apply the results of Section~\ref{sec:stationary} to determine the stationary regimes of the fluid dynamics associated with our timed Petri net model of emergency call center. As in Section~\ref{ex:fluid_dynamics}, we consider the subsystem reduced to the variables $z_1$, $z_5$ and $z_6$. The corresponding system of constraints over the germ variables $(u_1, \rho_1)$, $(u_5, \rho_5)$ and $(u_6, \rho_6)$ is given by:
\begin{subequations}\label{eq:germs}
\allowdisplaybreaks
\begin{align}
(\rho_1, u_1) & = 
\begin{multlined}[t]
\bigl(\rho_5 + \piu \rho_1 + \pia \rho_1, \\
N_1 + (u_5 - \rho_5 \taut) + \piu (u_1 - \rho_1 \tauu) + \pia (u_1 - \rho_1 \taua)\bigr)
\end{multlined} \label{eq:germs1} \\
(\rho_5, u_5) & =
\begin{cases}
\bigl(\rho_5,  N_2 + u_5 - \rho_5 (\taut + \taue') \bigr) 
\wedge \pie (\rho_1, u_1 - \rho_1 \taue)  & \text{if}\; \rho_6 = 0 \\[0.5ex]
\pie (\rho_1, u_1 - \rho_1 \taue) 
& \text{if}\; \rho_6 > 0
\end{cases} \label{eq:germs2} \\
(\rho_6, u_6) & = 
\bigl(\rho_6, N_2 - \rho_5 (\taut + \taue') + (u_6 - \rho_6 \tauu')\bigr)
\wedge \piu (\rho_1 , u_1 - \rho_1 \tauu)  \label{eq:germs3}
\end{align}
\end{subequations}
To solve this system, it is convenient to introduce the following quantity
\[
\taucyc := \pie (\taue + \taut) + \piu \tauu + \pia \taua \, ,
\]
which represents the average time of treatment of a call at level~1 of the model. Note that we exclude the trivial case where $\rho_1 = 0$ (and subsequently $\rho_5 = \rho_6 = 0$), since it cannot occur unless the quantity $N_1$ is null.

The $\rho$-part of~\eqref{eq:germs1} and~\eqref{eq:germs3} show that 
\[
\rho_5 = \pie \rho_1 \, , \qquad 0 \leq \rho_6 \leq \piu \rho_1 \, .
\] 

We start by considering the case where $\rho_6 = 0$. Since $\rho_1 > 0$, the minimum in~\eqref{eq:germs3} is necessarily attained by the left term. From this, we deduce 
\[
\rho_1 = \frac{N_2}{\pie (\taut + \taue')} \, .
\]
As $(\rho_5, u_5) \leq \pie (\rho_1, u_1 - \rho_1 \taue)$ (by~\eqref{eq:germs2}) and $\rho_5 = \pie \rho_1$, the inequality $u_5 \leq \pie (u_1 - \rho_1 \taue)$ holds. Using the $u$-part of~\eqref{eq:germs1}, we can show that this amounts to the inequality
\[
\frac{N_2}{N_1} \leq r_1 := \frac{\pie (\taut + \taue')}{\taucyc} \, .
\]

We now assume that $\rho_6 > 0$. The fact that $u_5 = \pie(u_1 - \rho_1 \taue)$ (by~\eqref{eq:germs2}) leads to the identity
\[
\rho_1 = \frac{N_1}{\taucyc} \, .  
\]
It remains to distinguish the subcases corresponding to the minimum in~\eqref{eq:germs3}. 
\begin{asparaitem}[\textbullet]
\item Suppose that the minimum is attained by the left term. We deduce that:
\[
\rho_6 = \frac{N_2}{\tauu'} - \frac{N_1}{\taucyc} \frac{\pie(\taut + \taue')}{\tauu'} = \frac{N_2 - N_1 r_1}{\tauu'} \, .
\]
Since $0 < \rho_6 \leq \piu \rho_1$, we also derive:
\[
r_1 < \frac{N_2}{N_1} \leq r_2 := \frac{\pie (\taut + \taue') + \piu \tauu'}{\taucyc} \, .
\]

\item If the minimum is reached by the right term, then we have $\rho_6 = \piu \rho_1$, or equivalently $\rho_6 = \piu \frac{N_1}{\taucyc}$. 
Moreover, we necessarily have $u_6 \leq N_2 - \rho_5 (\taut + \taue') + (u_6 - \rho_6 \tauu')$, which provides $\frac{N_2}{N_1} \geq r_2$. 
Note that the latter inequality is strict as soon as the minimum in~\eqref{eq:germs3} is attained by the right term only.
\end{asparaitem}

\begin{table}[t]
\caption{The normalized throughputs $\rho_1$, $\rho_5$ and $\rho_6$ as piecewise linear functions of $N_2 / N_1$.} \label{tab:throughputs}
\begin{center}
\setlength{\tabcolsep}{0.25cm}
\renewcommand{\arraystretch}{1.8}
\begin{tabular}{>{$\displaystyle}c<{$}>{$\displaystyle}c<{$}>{$\displaystyle}c<{$}>{$\displaystyle}c<{$}}
\toprule
& 0 \leq N_2/N_1 \leq r_1 & r_1 \leq N_2/N_1 \leq r_2 & r_2 \leq N_2/N_1 \\
\midrule
\rho_1 / \rho^* & \frac{\taucyc}{\pie (\taut + \taue')} \frac{N_2}{N_1} & 1 & 1 \\
\rho_5 / \rho^* & \frac{\taucyc}{\taut + \taue'} \frac{N_2}{N_1} & \pie & \pie \\
\rho_6 / \rho^* & 0 & \frac{\taucyc}{\tauu'} \Bigl(\frac{N_2}{N_1} - r_1\Bigr) & \piu \\[1ex]
\bottomrule
\end{tabular} 
\end{center}
\end{table}

To summarize, we report the possible values of the throughputs $\rho_1$, $\rho_5$ and $\rho_6$ in Table~\ref{tab:throughputs} in the stationary regimes. We normalize these values by a quantity $\rho^*$ which corresponds to the throughput (of transition $q_1$) in an ``ideal'' call center which involves as many level~2 operators as necessary, \ie~$N_2 = +\infty$. Then, the throughput $\rho^*$ is given by $N_1 / \taucyc$, where $\taucyc := \pie (\taue + \taut) + \piu \tauu + \pia \taua$ represents the average time of treatment at level~1. 

As shown in Table~\ref{tab:throughputs}, the ratios $\rho_1 / \rho^*$, $\rho_5 / \rho^*$ and $\rho_6 / \rho^*$ are piecewise linear functions of the ratio $N_2 / N_1$. The non-differentiability points are given by:
\[
r_1 := \frac{\pie (\taut + \taue')}{\taucyc} \qquad \quad
r_2 := \frac{\pie (\taut + \taue') + \piu \tauu'}{\taucyc} \, . 
\]
They separate three phases:
\begin{asparaenum}[(i)]
\item when $N_2/N_1$ is strictly smaller than $r_1$, the number of level~2 operators is so small that some extremely urgent calls cannot be handled, and no urgent call is handled. 
This 
is why the throughput of the latter calls at level~2 is null. 
Also, level~1 operators are slowed down by the congestion of level~2,
since, in the treatment of an extremely urgent call,
a level~1 operator cannot be released until the
call is handled by a level~2 operator. 

\item when $N_2/N_1$ is between $r_1$ and $r_2$, there are enough level~2 operators to handle all the extremely urgent calls,
which is why the throughput $\rho_5$ is equal to $\rho_1$ multiplied by the proportion $\pie$ of extremely urgent calls. As a consequence, level~2 is no longer slowing down level~1 (the throughput $\rho_1$ reaches its maximal value $\rho^*$). However, the throughput of urgent calls at level~2 is still limited because $N_2$ is not sufficiently large.
\item if $N_2/N_1$ is larger than $r_2$, the three throughputs reach their maximal values. This means that level~2 is sufficiently well-staffed w.r.t.~level~1.
\end{asparaenum}

This analysis provides a qualitative method to determine an optimal dimensioning of the system in stationary regimes. Given a fixed $N_1$, the number $N_2$ of level $2$ operators should be taken to be the minimal integer such that $N_2 / N_1 \geq r_2$. This ensures that the level $2$ properly handles the calls transmitted by the level $1$ (all calls are treated). Then, $N_1$ should be the minimal integer such that $\rho_1 = \frac{N_1}{\taucyc}$ dominates the arrival rate of calls. 

\begin{figure}[t]
\begin{center}
\includegraphics{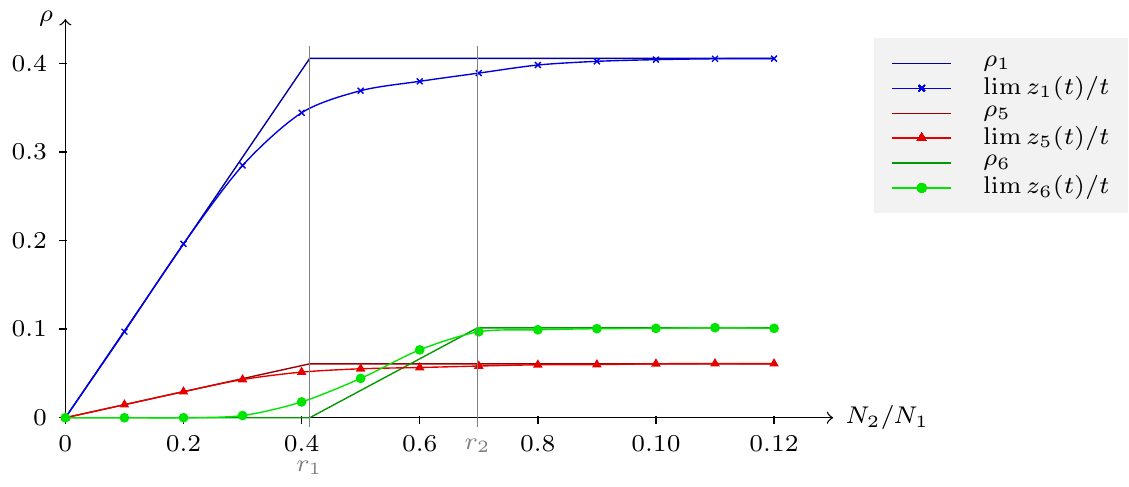}
\end{center}
\caption{Comparison of the throughputs of the non-fluid simulations with the theoretical throughputs (fluid model). The three phases are identified by two vertical lines.} \label{fig:phasesDiagramm} 
\end{figure}

\section{Experiments}\label{sec:experiments}

We finally compare the analytical results of Section~\ref{sec:application}, obtained in the fluid setting, with the asymptotic throughputs of the Petri net provided by simulations.

\subsection{Asymptotic behavior of the non-fluid dynamics.} 

We have implemented the $\delta$-discretization of the non-fluid dynamics (Equations~\eqref{eq:pnpriority1}--\eqref{eq:pnpriority3}, \eqref{eq:pnpriority5} and \eqref{eq:pnpriority4discrete}) since this setting is the closest to reality. Recall that, in this case, tokens are routed towards transitions $q_2$, $q_3$ and $q_4$ randomly according to a constant probability distribution. We assume that holding times are given by integer numbers of seconds, so that we take $\delta = 1\,\text{s}$. In this way, we compute the quantities $z_1(t)$, $z_5(t)$ and $z_6(t)$ by induction on $t \in \N$ using the equations describing the dynamics. In the simulations, we choose holding times and probabilities which are representative of the urgency of calls. 

Figure~\ref{fig:phasesDiagramm} compares the limits when $t \to +\infty$ of the throughputs $z_1(t)/t$, $z_5(t)/t$, $z_6(t)/t$ of the ``real'' system, with the throughputs $\rho_1$, $\rho_5$ and $\rho_6$ of the stationary solutions which have been determined in Section~\ref{sec:application}. The latter are simply computed using the analytical formul{\ae} of Table~\ref{tab:throughputs}. We estimate the limits of the throughput $z_i(t)/t$ by evaluating the latter quantity for
$t = 10^6 \,\text{s}$. As shown in Figure~\ref{fig:phasesDiagramm}, these estimations confirm the existence of three phases, as described in the previous section. The convergence of $z_i(t)/t$ towards the throughputs $\rho_i$ is mostly reached in the two extreme phases. In the intermediate phase, the difference between the limit of $z_i(t)/t$ and the throughput $\rho_i$ is more important. 
This originates from the stochastic nature of the routing, which causes more variations in the realization of the minima in the $z_i(t)$: the throughput of $q_6$ increases and the throughputs of $q_1$ and $q_5$ decrease.

\subsection{Asymptotic behavior of the fluid dynamics.} We have also simulated the discrete-time fluid dynamics (using Equations~\eqref{eq:pnpriority1}--\eqref{eq:pnpriority3} and~\eqref{eq:pnpriority5}--\eqref{eq:pnpriority6fluid}). All simulations have been computed with exact rationals in $\mathbb{Q}$.

In most cases, we observe that the corresponding asymptotic throughputs converge to the throughputs of the stationary solutions. This is illustrated in Figures~\ref{fig:cvgFluid}(a), \ref{fig:cvgFluid}(b) and~\ref{fig:cvgFluid}(c), which are obtained using the same set of holding times, and by varying the ratio $N_2 / N_1$ (lower, intermediate and upper phase respectively). 

However, there are also cases in which the convergence does not hold. In the experiments we have made, this happens only in the lower phase and in the intermediate phase, that is, when $N_2/N_1 < r_2$. This is illustrated in Figure~\ref{fig:cvgFluid}(d), in which we have increased $\taue'$ by one unit of time in comparison to Figure~\ref{fig:cvgFluid}(b). Such cases suggest the existence of other kinds of stationary regimes of the dynamics, in which the system oscillates between different phases. We remark that the non-convergence appears to be related to the existence of arithmetical relationships between the holding times of places. An interpretation lies in the fact that, if cycle times are not coprime in the system, phenomena of synchronization may lead to recurrent slow-down of extreme urgent calls by urgent calls in the two lower phases, which could lower the throughput of the system.

\begin{figure}[t]
\begin{center}
\includegraphics{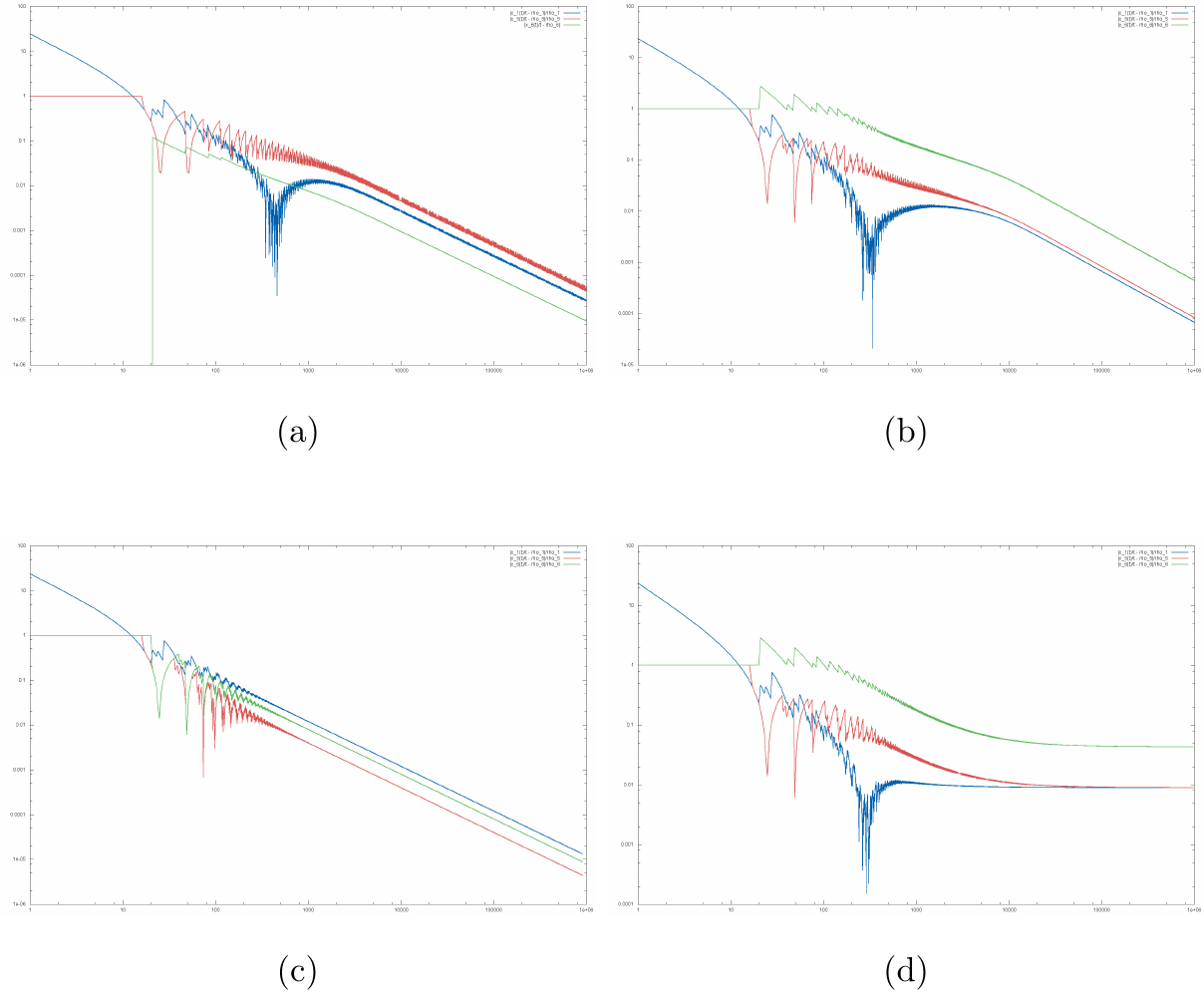}
\caption{Comparison of the fluid dymamics with the stationary regimes. Error ratios $|z_i(t)/t - \rho_i| / \rho_i$ are plotted in log-log scale, respectively in blue, red and green when $i = 1, 5, 6$.}\label{fig:cvgFluid}
\end{center}
\end{figure}

\section{Concluding remarks}

We have shown that timed Petri nets with free choice and priority routing can be analyzed by means of tropical geometry. This allows us to identify the congestion phases
in the fluid version of the dynamics of the Petri net.
We have applied this method to a model of emergency call center.
Numerical experiments indicate that these theoretical results are representative of the real dynamics. 

In future work, we aim at comparing the behaviors of the fluid deterministic model and of the discrete stochastic one. We also plan to study uniqueness conditions of
the stationary regimes, and conditions under which convergence of the fluid dynamics to a stationary regime can be shown. We will refine our Petri net model of emergency call center to take care of the heterogeneous nature of level~2 (calls to police and firemen require different instruction times). To this end, it will be helpful to implement an analysis tool determining automatically the stationary regimes of a timed Petri net given in input. Finally, we plan to analyze the treatment times of the system, on top of the throughputs.

\bibliographystyle{alpha}

\begin{thebibliography}{HOvdW06}

\bibitem[ABG15]{ABG2015}
X.~Allamigeon, V.~B{\oe}uf, and S.~Gaubert.
\newblock Performance evaluation of an emergency call center: tropical
  polynomial systems applied to timed petri nets.
\newblock In {\em FORMATS'15}, volume 9268 of {\em Lecture Notes in Computer
  Science}. Springer, 2015.

\bibitem[AFG{\etalchar{+}}14]{uli2013}
X.~Allamigeon, U.~Fahrenberg, S.~Gaubert, A.~Legay, and R.~Katz.
\newblock Tropical {F}ourier-{M}otzkin elimination, with an application to
  real-time verification.
\newblock {\em International Journal of Algebra and Computation},
  24(5):569--607, 2014.

\bibitem[AN01]{AbdullaNylen01}
P.~A. Abdulla and A.~Nylén.
\newblock Timed {P}etri nets and {BQO}s.
\newblock In {\em Applications and Theory of Petri Nets '01}, volume 2075 of
  {\em LNCS}. Springer, 2001.

\bibitem[BCH{\etalchar{+}}05]{Berard05}
B.~Bérard, F.~Cassez, S.~Haddad, D.~Lime, and O.~H. Roux.
\newblock Comparison of the expressiveness of timed automata and time {P}etri
  nets.
\newblock In {\em FORMATS'05}, volume 3829 of {\em Lecture Notes in Computer
  Science}. Springer, 2005.

\bibitem[BCOQ92]{bcoq}
F.~Baccelli, G.~Cohen, G.J. Olsder, and J.P. Quadrat.
\newblock {\em Synchronization and Linearity}.
\newblock Wiley, 1992.

\bibitem[BD91]{BerthomieuDiaz91}
B.~Berthomieu and M.~Diaz.
\newblock Modeling and verification of time dependent systems using time
  {P}etri nets.
\newblock {\em Software Engineering, IEEE Transactions on}, 17(3), 1991.

\bibitem[BJS09]{tapaal}
J.~Byg, K.~Y. Jørgensen, and J.~Srba.
\newblock Tapaal: Editor, simulator and verifier of timed-arc {P}etri nets.
\newblock In {\em ATVA'09}, volume 5799 of {\em LNCS}. Springer, 2009.

\bibitem[BV06]{tina}
B.~Berthomieu and F.~Vernadat.
\newblock Time {P}etri nets analysis with {TINA}.
\newblock In {\em QEST'06}. IEEE, 2006.

\bibitem[CGQ95]{CGQ95b}
G.~Cohen, S.~Gaubert, and J.P. Quadrat.
\newblock Asymptotic throughput of continuous timed {P}etri nets.
\newblock In {\em 34th Conference on Decision and Control}, 1995.

\bibitem[CGQ98]{CGQ95a}
G.~Cohen, S.~Gaubert, and J.P. Quadrat.
\newblock Algebraic system analysis of timed {P}etri nets.
\newblock In J.~Gunawardena, editor, {\em Idempotency}, Publications of the
  Isaac Newton Institute, pages 145--170. Cambridge University Press, 1998.

\bibitem[CHB14]{cottenceau}
B.~Cottenceau, L.~Hardouin, and J.L. Boimond.
\newblock Modeling and control of weight-balanced timed event graphs in dioids.
\newblock {\em IEEE Trans. Autom. Control}, 59(5), 2014.

\bibitem[FGQ11]{farhi}
N.~Farhi, M.~Goursat, and J.-P. Quadrat.
\newblock Piecewise linear concave dynamical systems appearing in the
  microscopic traffic modeling.
\newblock {\em Linear Algebra and Appl.}, 2011.

\bibitem[GG98]{gg0}
S.~Gaubert and J.~Gunawardena.
\newblock The duality theorem for min-max functions.
\newblock {\em C. R. Acad. Sci. Paris.}, 326, S\'erie I:43--48, 1998.

\bibitem[GRR04]{Gardey04}
G.~Gardey, O.~H. Roux, and O.~F. Roux.
\newblock Using zone graph method for computing the state space of a time
  {P}etri net.
\newblock In {\em FORMATS'04}, volume 2791 of {\em Lecture Notes in Computer
  Science}. Springer, 2004.

\bibitem[HOvdW06]{how06}
G.~Heidergott, G.~J. Olsder, and J.~van~der Woude.
\newblock {\em Max Plus at work}.
\newblock Princeton University Press, 2006.

\bibitem[II12]{inoue}
R.~Inoue and S.~Iwao.
\newblock Tropical curves and integrable piecewise linear maps.
\newblock In C.~Athorne, D.~Maclagan, and I.~Strachan, editors, {\em Tropical
  geometry and integrable systems}, volume 580 of {\em Contemporary
  mathematics}, pages 21--40. AMS, 2012.

\bibitem[JJMS11]{Jacobsen11}
L.~Jacobsen, M.~Jacobsen, M.~H. Møller, and J.~Srba.
\newblock Verification of timed-arc {P}etri nets.
\newblock In {\em SOFSEM'11}, volume 6543 of {\em LNCS}. Springer, 2011.

\bibitem[Lib96]{Libeaut}
L.~Libeaut.
\newblock {\em Sur l'utilisation des dio\"\i des pour la commande des
  syst\`emes \`a \'ev\'enements discrets}.
\newblock Th\`ese, \'Ecole Centrale de Nantes, 1996.

\bibitem[LRST09]{romeo}
D.~Lime, O.~H. Roux, C.~Seidner, and L.-M. Traonouez.
\newblock Romeo: A parametric model-checker for {P}etri nets with stopwatches.
\newblock In {\em TACAS'09}, volume 5505 of {\em Lecture Notes in Computer
  Science}. Springer, 2009.

\bibitem[Plu99]{maxplusblondel}
M.~Plus.
\newblock Max-plus-times linear systems.
\newblock In {\em Open Problems in Mathematical Systems and Control Theory}.
  Springer, 1999.

\bibitem[RR15]{reboulraclot}
S.~Raclot and R.~Reboul.
\newblock Analysis of the new call center organization at {PP} and {BSPP}.
\newblock Personal communication to the authors, 2015.

\bibitem[Srb08]{Srba08}
J.~Srba.
\newblock Comparing the expressiveness of timed automata and timed extensions
  of {P}etri nets.
\newblock In {\em FORMATS'08}, volume 5215 of {\em LNCS}. Springer, 2008.

\end{thebibliography}
\newcommand{\etalchar}[1]{$^{#1}$}

\appendix
\section{Proof of Theorem~\ref{th:semantics}}\label{app:semantics}

\begin{lemma}\label{lemma:trace}
Suppose that all the holding times $\tau_p$ ($p \in \Pcal$) are positive, and consider the subpart of the execution trace formed by the transitions fired at the instant $t$, \ie:
\begin{equation}
\dots \trans[d] \sigma^{t^-} \trans[q_1] \; \trans[q_2] \dots\trans[q_n] \sigma^{t} \trans[d'] \label{eq:tracebis}
\end{equation}
(with $d > 0$ unless $t = 0$, and $d' > 0$). Then the following two properties hold:
\begin{enumerate}[(i)]
\item for all $p \in \Ppriority$, no transition $\trans[p\outm]$ can occur before a transition $\trans[p\outp]$ in~\eqref{eq:tracebis};
\item any pair of consecutive of consecutive transitions $\trans[q_i] \trans[q_{i+1}]$ can be switched in~\eqref{eq:tracebis} without changing the states occurring after, provided that $(q_i, q_{i+1})$ is not equal to $(p\outp, p\outm)$ for some $p \in \Ppriority$.
\end{enumerate}
\end{lemma}

\begin{proof}
\begin{asparaenum}[(i)]
\item Suppose that a transition $\trans[p\outm]$ occurs before $\trans[p\outp]$ in~\eqref{eq:tracebis}, \ie~we have a subsequence of the form $\sigma \trans[p\outm] \sigma' \dots \trans[p\outp]$. As all the holding times are positive, all the tokens consumed by $p\outp$ are already present in the state $\sigma$. In other words, the transition $\sigma \trans[p\outp] \dots$ is valid in the semantics. This contradicts the priority rule.

\item Consider a pair of consecutive transitions 
\[
\sigma^i \trans[q_i] \trans[q_{i+1}] \sigma^{i+2}
\]
such that $(q_i, q_{i+1}) \neq (p\outp, p\outm)$. As discussed in the previous case, since the holding times are positive, the transition $q_{i+1}$ does not consume tokens produced by the transition $q_i$. Besides, there is no priority rule between $q_i$ and $q_{i+1}$. Therefore, $q_{i+1}$ can be fired before $q_i$, and the sequence $\sigma^i \trans[q_{i+1}] \trans[q_i]$ leads to the same state $\sigma^{i+2}$. It follows that all the subsequent states remain identical. \qedhere
\end{asparaenum}
\end{proof}

Let us now prove Theorem~\ref{th:semantics}. It is useful to extract the part of the execution trace leading to the state $\sigma^t$ of the Petri net at the instant $t$. It has one of the following two forms:
\begin{equation}
\dots \trans[d] \sigma^{t^-} \trans[q_1] \; \trans[q_2] \dots\trans[q_n] \sigma^t \trans[d'] \label{eq:trace1}
\end{equation}
on
\begin{equation}
\dots \trans[d] \; \trans[q_1] \; \trans[q_2] \dots \trans[q_n] \; \trans[d''] \sigma^t \trans[d'] \label{eq:trace2}
\end{equation}
depending on whether some transitions $q \in \Qcal$ are fired at the instant $t$ or not. In both cases, $d', d''> 0$ and $d > 0$ unless $t = 0$, and the durations of the time-elapsing transitions occurring before the state $\sigma^t$ in the trace sum up to $t$. 

In this context, $z_q(t)$ counts the number of transitions occurring before $\sigma^t$ in the trace, while $x_p(t)$ is given by the sum of $M_p$ and the number of transitions $q \in p\inc$ occurring before $\sigma^t$. The constraint~\eqref{eq:pnpriority1} is therefore trivially satisfied by definition of $x_p(t)$ and the $z_q(t)$. 

Consider $p \in \Pconflict$. Since we use the earliest behavior semantics and $p$ is free choice, any token of the initial marking is consumed at the instant $0$, and any token brought by an upstream transition $q' \in p\inc$ at the instant $s \geq 0$ is consumed at the instant $s + \tau_p$. 
As a consequence, we can build a bijection which maps each initial token with the transition $\trans[q]$ which consumes it at the instant $0$, and any transition $\trans[q']$ occurring at the instant $s - \tau_p$ with the transition $\trans[q]$ which consumes at the instant $s$ the token brought by $q'$ to place $p$.
We deduce that 
\[
M_p +  \sum_{q' \in p\inc} z_{q'}(t - \tau_p) = \sum_{q \in p\out} z_q(t) \, .
\]
Using the constraint~\eqref{eq:pnpriority1}, this yields to $x_p(t - \tau_p) = \sum_{q \in p\out} z_q(t)$.

Now, let us take $q \in \Qsync$. 
Consider $p \in q\inc$. Recall that, by definition of $\Qsync$, $q$ is the only downstream transition of place $p$ if $q$. Therefore, every transition $\trans[q]$ arising at the instant $s$ consumes a token from place $p$. This token is either a initial token from $M_p$, or a token brought by a transition $q' \in p\inc$ fired before the instant $s - \tau_p$ (included). 
Therefore, we have $z_q(t) \leq x_p(t - \tau_p)$. In fact, $x_p(t - \tau_p) - z_q(t)$ is equal to the number of tokens with age greater than or equal to $\tau_p$ located in place $p$ in the state $\sigma^t$. At the instant $t + \eps$ with $0 < \eps < d'$, the age of these tokens will be strictly greater than $\tau_p$. 
Therefore, if $x_p(t - \tau_p) - z_q(t) > 0$ for all $p \in q\inc$, the transition $q$ can be fired at the instant $t + \eps$. 
But this is impossible in the earliest behavior semantics, since the places $p \in q\inc$ are not allowed to contain tokens with age strictly greater than $\tau_p$ while their downstream transition $q$ can be fired. We deduce that $x_p(t - \tau_p) = z_q(t)$ for some $p \in q\inc$. This proves~\eqref{eq:pnpriority3}.

Finally, consider $p \in \Ppriority$. Using similar arguments as the ones used in the previous case, we can show that $z_{p\outp}(t) \leq x_r(t-\tau_r)$ for all $r \in (p\outp)\inc$, $r \neq p$, and $z_{p\outm}(t) \leq x_r(t-\tau_r)$ for all $r \in (p\outm)\inc$, $r \neq p$. Besides, we have $z_{p\outp}(t) + 
z_{p\outm}(t) \leq x_p(t - \tau_p)$, since every firing of the transition $p\outp$ or $p\outm$ at the instant $s$ consumes a token of $M_p$ or a token brought by an upstream transition of $p$ before the instant $s - \tau_p$. In consequence, as the function $z_{p\outm}$ is non-decreasing, we obtain:
\[
z_{p\outp}(t) + z_{p\outm}(t^-) \leq x_p(t - \tau_p) \, .
\]
In order to prove that~\eqref{eq:pnpriority4} is satisfied, we distinguish two cases depending on the form of the trace:
\begin{asparaenum}[(i)]\item if the trace is of the form~\eqref{eq:trace1}, then, by Lemma~\ref{lemma:trace}, we can rewrite the subpart of the trace as follows:
\[
\dots \trans[d] \sigma^{t^-} \trans[q'_1] \; \trans[q'_2] \dots \trans[q'_k] \; \underbrace{\trans[p\outp] \dots \trans[p\outp]}_{k_+ \ \text{times}} \sigma \underbrace{\trans[p\outm] \dots \trans[p\outm]}_{k_- \ \text{times}} \sigma^t \trans[d']
\]
where $k_+, k_- \geq 0$, and where $p\outp$ and $p\outm$ do not appear in the $q'_i$. 
Then, the quantity $x_p(t - \tau_p) - z_{p\outp}(t) - z_{p\outm}(t^-)$ corresponds to the number of tokens with age greater than or equal to $\tau_p$ in the intermediary state $\sigma$. 
If it is positive, and if $x_r(t-\tau_r) - z_{p\outp}(t) > 0$ for all $r \in (p\outp)\inc$ such that $r \neq p$, then the transition $p\outp$ can be fired right after the state $\sigma$. 
This contradicts the priority rule if $k_- > 0$. 
If $k_- = 0$, we can fire $p\outp$ at the instant $t + \eps$ ($0 < \eps < d'$), which contradicts the definition of the earliest behavior semantics (all the upstream place of $p\outp$ contains a token older than allowed).

\item if the trace is of the form~\eqref{eq:trace2}, then $z_{p\outm}(t^-) = z_{p\outm}(t)$. In this case, the quantity $x_p(t - \tau_p) - z_{p\outp}(t) - z_{p\outm}(t^-)$ represents the number of tokens with age greater than or equal to $\tau_p$ at place $p$ in the state $\sigma^t$. If $x_p(t - \tau_p) - z_{p\outp}(t) - z_{p\outm}(t^-) > 0$ and $x_r(t-\tau_r) - z_{p\outp}(t) > 0$ for all $r \in (p\outp)\inc$ such that $r \neq p$, then the transition $p\outp$ can be fired at the instant $t + \eps$ with $0 < \eps < d'$. This is again a contradiction with the earliest behavior semantics. 
\end{asparaenum}

In both cases, we have $x_p(t - \tau_p) - z_{p\outp}(t) - z_{p\outm}(t^-) = 0$ or $x_r(t-\tau_r) - z_{p\outp}(t) = 0$ for some $r \in (p\outp)\inc$ such that $r \neq p$. We deduce that the constraint~\eqref{eq:pnpriority4} holds.

Now assume that $x_p(t - \tau_p) - z_{p\outp}(t) - z_{p\outm}(t) > 0$ and $x_r(t-\tau_r) - z_{p\outm}(t) > 0$ for all $r \in (p\outm)\inc$ such that $r \neq p$. 
These quantities correspond to the number of tokens in places $p$ and $r$ with age greater than or equal to $\tau_p$ and $\tau_r$ respectively, in the state $\sigma^t$. 
Thus, the transition $p\outm$ is activated at the instant $t + \eps$ for all $\eps > 0$ sufficiently small. 
Note that $x_p(t - \tau_p) - z_{p\outp}(t) - z_{p\outm}(t^-) > 0$ as $z_{p\outm}(t^-) < z_{p\outm}(t)$. 
Thus, there exists a place $r' \in (p\outp)\inc$ with $r \neq p$, such that $x_{r'}(t-\tau_{r'}) = z_{p\outp}(t)$. 
In other words, place $r'$ does not contain any token with age greater than or equal to $\tau_{r'}$.  
Given $\eps > 0$ sufficiently small, this is still true at the instant $t + \eps$, so that the transition $p\outp$ cannot be fired at the instant $t + \eps$. 
Therefore, we are allowed to fire the transition $p\outm$ at the instant $t + \eps$, which is a contradiction with the definition of the earliest behavior semantics. As a result, \eqref{eq:pnpriority5} is satisfied.\qed

\section{Proof of Theorem~\ref{th:stationary}}\label{app:stationary}

We denote the lexicographic order over $\R^2$ by $\preceq$, and we use the notation $x \prec y$ when $x \preceq y$ and $x \neq y$.

We first remark that given two functions $f$ and $g$ of affine germs $(\rho, u)$ and $(\rho', u')$, the function $t \mapsto f(t) \wedge g(t)$ belongs to the affine germ given by the minimum $(\rho, u)\wedge (\rho', u')$ taken in the lexicographic order. Besides, if $\tau \in \delta \N$, the function $t \mapsto f(t - \tau)$ belongs to the affine germ $(\rho, u - \tau \rho)$. Finally, for all $\lambda \in \R$, the affine germ of the map $t \mapsto \lambda f(t)$ is equal to $\lambda (\rho, u) = (\lambda \rho, \lambda u)$.

Now, suppose that $x_p$ and $z_q$ are stationary solutions of the $\delta$-discretization of the fluid dynamics, and let $(\rho_p, u_p)$ and $(\rho_q, u_q)$ be the respective germs, for $p \in \Pcal$ and $q \in \Qcal$. 
Since the functions $x_p$ and $z_q$ satisfy the constraints~\eqref{eq:pnpriority1}, \eqref{eq:pnpriority3}, \eqref{eq:pnpriority5}, \eqref{eq:pnpriority6fluid} for all $t \in \delta \N$, we deduce from the previous properties that the constraints~\eqref{eq:fixpoint_germ1}--\eqref{eq:fixpoint_germ4} are satisfied. Besides, given $p \in \Ppriority$, \eqref{eq:pnpriority4discrete} ensures that for all $t \in \delta \N$, we have:
\[
z_{p\outp}(t) = \bigl(x_p(t-\tau_p) - z_{p\outm}(t - \delta) \bigr) \wedge \bigwedge_{r \in (p\outp)\inc \, , \,  r \neq p} x_r(t - \tau_r) \, .
\]
Consequently, we obtain
\begin{equation}
(\rho_{p\outp}, u_{p\outp}) = (\rho_p - \rho_{p\outm}, u_p - \rho_p \tau_p - u_{p\outm} + \rho_{p\outm} \delta) \wedge \bigwedge_{r \in (p\outp)\inc \, , \,  r \neq p} (\rho_r, u_r - \rho_r \tau_r) \, .
\label{eq:fixpoint_germ_app}
\end{equation}
If $\rho_{p\outm} = 0$, this amounts to the constraint given in~\eqref{eq:fixpoint_germ5}. Now consider the case where $\rho_{p\outm} > 0$. Using~\eqref{eq:fixpoint_germ4}, we know that 
\[
(\rho_{p\outm}, u_{p\outm}) \preceq (\rho_p - \rho_{p\outp}, u_p - \rho_p \tau_p - u_{p\outp}) \, ,
\]
and thus
\[
(\rho_{p\outp}, u_{p\outp}) \preceq (\rho_p - \rho_{p\outm}, u_p - \rho_p \tau_p - u_{p\outm}) \, ,
\]
Since $\delta > 0$, it follows that 
\[
(\rho_{p\outp}, u_{p\outp}) \prec (\rho_p - \rho_{p\outm}, u_p - \rho_p \tau_p - u_{p\outm} + \rho_{p\outm} \delta) \, .
\]
We conclude that the constraint~\eqref{eq:fixpoint_germ_app} is equivalent to~\eqref{eq:fixpoint_germ5} when $\rho_{p\outm} > 0$.

Conversely, let $(\rho_p, u_p)$ and $(\rho_q, u_q)$ be solutions of System~\eqref{eq:fixpoint_germ}. We define $x_p$ and $z_q$ as the functions given by $x_p(t) = u_p + \rho_p k \delta$ and $z_q(t) = u_q + \rho_q k \delta$ for all $t \in [k\delta, (k+1) \delta)$ and $k \in \N$. 
The constraints~\eqref{eq:fixpoint_germ1}--\eqref{eq:fixpoint_germ4} ensure that~\eqref{eq:pnpriority1}, \eqref{eq:pnpriority3}, \eqref{eq:pnpriority5}, \eqref{eq:pnpriority6fluid} hold for all $t \in \delta \N$. 
Since all the holding times $\tau_p$ belong to $\delta \N$ and the functions $x_p$ and $z_q$ are constant on the intervals of the form $[k\delta, (k+1) \delta)$, we deduce that these constraints~\eqref{eq:pnpriority1}, \eqref{eq:pnpriority3}, \eqref{eq:pnpriority5}, \eqref{eq:pnpriority6fluid} actually hold for all $t$ in such intervals, and so for all $t \geq 0$. 
Moreover, as previously shown, the constraint given in~\eqref{eq:fixpoint_germ5} is equivalent to~\eqref{eq:fixpoint_germ_app}, since~\eqref{eq:fixpoint_germ4} is satisfied. 
This proves that the constraint~\eqref{eq:pnpriority4discrete} holds for all $t \in \delta \N$. 
It remains to show that the latter constraint is satisfied when $t \in (k\delta, (k+1) \delta)$. First observe that $z_{p\outp}(k \delta) \leq \bigwedge_{r \in (p\outp)\inc \, , \,  r \neq p} x_r(k\delta - \tau_r)$ ensures that $z_{p\outp}(t) \leq \bigwedge_{r \in (p\outp)\inc \, , \,  r \neq p} x_r(t - \tau_r)$. 
Besides, by~\eqref{eq:pnpriority5}, we know that $z_{p\outp}(t) \leq x_p(t-\tau_p) - z_{p\outm}(t)$.
We now distinguish two cases:
\begin{asparaenum}[(i)]
\item if we have $z_{p\outp}(k \delta) = \bigwedge_{r \in (p\outp)\inc \, , \,  r \neq p} x_r(k\delta - \tau_r)$, then straightforwardly, $z_{p\outp}(t) = \bigwedge_{r \in (p\outp)\inc \, , \,  r \neq p} x_r(t - \tau_r)$.
\item if $z_{p\outp}(k \delta) = x_p(k\delta-\tau_p) - z_{p\outm}((k - 1)\delta)$, we obtain:
\[ 
z_{p\outp}(k \delta) \geq x_p(k\delta-\tau_p) - z_{p\outm}(k\delta) \geq z_{p\outp}(k\delta) \, , 
\]
where the first inequality comes from the fact that $z_{p\outm}$ is non-decreasing, and the second inequality from~\eqref{eq:pnpriority5}. We deduce that $z_{p\outp}(k \delta) = x_p(k\delta-\tau_p) - z_{p\outm}(k\delta)$. Hence, we get $z_{p\outp}(t) = x_p(t - \tau_p) - z_{p\outm}(t)$.
\end{asparaenum}

As a consequence, in both cases, we have proved that $z_{p\outp}(t)$ is the minimum between $z_{p\outp}(t) - z_{p\outm}(t)$ and $\bigwedge_{r \in (p\outp)\inc \, , \,  r \neq p} x_r(t - \tau_r)$. This shows that~\eqref{eq:pnpriority4discrete} holds for all $t \geq 0$. \qed

\end{document}